\documentclass[11pt]{article}    
\usepackage[margin=1in]{geometry}
\usepackage{amsmath}
\usepackage{amsthm}
\usepackage{amssymb}
\DeclareMathOperator{\dom}{dom}
\newtheorem{prop}{Proposition}[section]
\newtheorem{theo}[prop]{Theorem}
\newtheorem{corol}[prop]{Corollary}

\title{Unfriendly partitions when avoiding vertices of finite degree}
\author{Leandro Fiorini Aurichi$^*$ \and Lucas Silva Sinzato Real\footnote{Instituto de Ci\^{e}ncias Matem\'{a}ticas e de Computa\c{c}\~{a}o, Universidade de São Paulo, S\~{a}o Carlos, S\~{a}o Paulo, Brazil}}
\date{April 2023}
\begin{document}
\maketitle
\begin{abstract}
An unfriendly partition of a graph $G = (V,E)$ is a function $c: V \to 2$ such that $|\{x\in N(v): c(x)\neq c(v)\}|\geq |\{x\in N(v): c(x)=c(v)\}|$ for every vertex $v\in V$, where $N(v)$ denotes its neighborhood. It was conjectured by Cowen and Emerson \cite{cowen} that every graph has an unfriendly partition, but Milner and Shelah in \cite{shelah_milner_1990} found counterexamples for that statement by analyzing graphs with uncountably many vertices. Curiously, none of their graphs have vertices with finite degree. Therefore, as a natural direction to approach, in this paper we search for the least cardinality of a graph with that property that admits no unfriendly partitions. Actually, among some other independence results, we conclude that this size cannot be determined from the usual axioms of set theory. 
\end{abstract}

\section{Introduction and problem background:}\label{s1}
\paragraph{}

For a given graph $G = (V,E)$, the neighborhood of a vertex $v\in V$ will be denoted by $N(v) = \{u\in V : uv \in E\}$, so that its degree is the cardinal $|N(v)|$. Besides that, a function $c : D \to 2$ is called a \textbf{partial coloring} of the graph, where $2 = \{0,1\}$ is a set of two elements (called \textbf{colors}) and $D\subset V$ is any subset. In this paper, we say that a partial coloring $c: D \to 2$ is \textbf{unfriendly} in a vertex $v\in D$ if $|\{x\in N(v)\cap D : c(x) \neq c(v)\}|\geq |\{x\in N(v)\cap D: c(x)=c(v)\} \cup (D\setminus N(v))|$. Then, $c$ is said to be an \textbf{unfriendly partial coloring} if it is unfriendly in all the vertices of its domain. With that definition, it follows directly that any extension $c' : D' \to 2$ of $c$ is unfriendly in a vertex $v\in D$ for which $c$ is unfriendly. Finally, an \textbf{unfriendly partition} is an unfriendly coloring $c: D \to 2$ defined in the whole graph, that is, $D= V$.

In that terms, we first highlight that every finite graph admits an unfriendly partition. This because, if $G = (V,E)$ is such a graph, an unfriendly  partition can be taken to be a coloring $c: V \to 2$ that maximizes the cardinality $|\{uv \in E : c(u)\neq c(v)\}|$. Consequently, compactness arguments verify that every locally finite graph (i.e., whose vertices have all finite degree) has also an unfriendly partition. Actually, as shown by Aharoni, Milner and Prikry in \cite{Aharoni1990}, every graph with only finitely many vertices of infinite degree admits an unfriendly partition.

Supported by these observations, it is natural to ask if every graph has an unfriendly partition, as done by Cowen and Emerson in \cite{cowen}. In 1990, Milner and Shelah in \cite{shelah_milner_1990} answered that problem negatively by exhibiting a family of uncountable graphs in which these colorings cannot be found. Further than being the only known graphs that admit no unfriendly partitions, their counterexamples have some other peculiarities. For example, the smallest member of this family has $\mathfrak{c}^{+\omega}$ vertices, none of them of finite degree. Under that notation, $\mathfrak{c}^{+\omega}$ refers to the first limit cardinal greater than $\mathfrak{c}$.

In particular, even if the Continnum Hypothesis ($\mathrm{CH}$) is assumed, $\aleph_{\omega}$ vertices are needed to describe the least known graph that has no unfriendly partition. Actually, Milner and Shelah also verified in \cite{shelah_milner_1990} that this amount can be obtained in some theories where the continuum has no fixed size. On the other hand, by previous results in the literature that we will revisit on Section \ref{apendice}, a counterexample with fewer than $\aleph_{\omega}$ vertices cannot be constructed when vertices of finite degree are forbidden. Hence, in this paper we investigate if, indeed, $\mathrm{ZFC}$ proves that $\aleph_{\omega}$ is the least cardinal $\varkappa$ such that there is a graph $G = (V,E)$ with $\varkappa$ vertices, all them of infinite degree, having no unfriendly partition.

More than that, we work under $\mathrm{ZFC} + \aleph_{\omega}<\mathfrak{c}$ to discuss whether the statements $\varkappa = \aleph_{\omega}$ and $\varkappa = \mathfrak{c}^{+\omega}$ hold. In this stronger theory, we observe that the cardinals $\aleph_{\omega}$ and $\mathfrak{c}^{+\omega}$ are distinct. Even so, we prove the following independence assertions as our main result:       

\begin{theo}
	\label{nosso}
	Let $\varkappa$ be the least cardinal such that there is a graph $G = (V,E)$ without vertices of finite degree, with $|V| = \varkappa$ and that admits no unfriendly partition. Then, the following statements are independent from $\mathrm{ZFC}+\aleph_{\omega}<\mathfrak{c}$:
	\begin{enumerate}
		\item $\varkappa = \aleph_{\omega}$.
		\item $\varkappa = \mathfrak{c}^{+\omega}$, where, as before, $\mathfrak{c}^{+\omega}$ is the least limit cardinal greater than $\mathfrak{c}$. 
	\end{enumerate}
	In particular, those statements are independent from the usual axioms of set theory.
\end{theo}
 
 Hence, Theorem \ref{nosso} is an independence result that arises from a problem related to Infinite Graphs. Lying in the intersection between Graph and Set Theory, thus, this paper is written detailing as many arguments as possible, in an attempt to be self-contained for readers from both (but not limited to these) areas. Therefore, depending on them, some proofs from Section \ref{apendice} may seem longer than they could be.

In any case, $\varkappa$ is far from being $\aleph_0$, so that the Unfriendly Partition Conjecture is now the question whether every countable graph has an unfriendly partition. However, even if $\aleph_0 < \kappa < \aleph_{\omega}$ or $\kappa = \mathfrak{c}$, counterexamples of size $\kappa$ are also unknown. According to Theorem \ref{nosso}, incidentally, such a graph $G$ must have at least $\min\{|N(v)| : v \in V(G), |N(v)|\geq \aleph_0\}$ vertices of finite degree.

\section{Possible values for $\varkappa$}\label{apendice}
\paragraph{}
In this section, we present the consistence results of Theorem \ref{nosso}, related to the existence of unfriendly partitions in graphs that have no vertices of finite degree. As pointed out in the introduction, this is motivated by the constructions of Milner and Shelah in \cite{shelah_milner_1990} of the only known graphs admitting no unfriendly partitions. Below, we recover the properties of these graphs:

\begin{theo}[\cite{shelah_milner_1990}, Theorem 2]
	\label{contraexemplozfc}
	Given an infinite cardinal $\lambda$, let $\alpha$ be the ordinal such that $|\wp(\lambda)| = 2^{\lambda} = \aleph_{\alpha}$. Then, there is a graph $G = (V,E)$ without vertices of finite degree, having no unfriendly partitions and such that $|V| = (2^{\lambda})^{+\omega}$, where $(2^{\lambda})^{+\omega} = \aleph_{\alpha+\omega}$. 
\end{theo}

 In particular, by taking $\lambda = \omega$ in the theorem above, we conclude that, forbidding vertices of finite degree, at least $(2^{\omega})^{+\omega} = \mathfrak{c}^{+\omega}$ vertices were needed to find a graph that has no unfriendly partitions. Since $\aleph_{\omega}\leq \mathfrak{c}^{+\omega}$, that amount of vertices is quite surprising. Therefore, we will study the least cardinal $\varkappa$ that satisfies the following property: there is a graph with $\varkappa$ vertices, all them of infinite degree, that admits no unfriendly partition. The current literature concerning unfriendly partitions easily provides a first lower bound for its value:
 
 \begin{theo}[\cite{Aharoni1990}, Theorem 2]
 	\label{finitoscardinais}
 	Let $\kappa_0 < \kappa_1 < \dots <\kappa_n$ be a finite collection of infinite cardinals, with $\kappa_i$ regular for all $1 \leq i \leq n$. Let $G = (V,E)$ be a graph such that $|\{v\in V : |N(v)| \text{ is finite}\}|< \kappa_0$ and $|N(v)| \in \{\kappa_0, \kappa_1,\dots,\kappa_n\}$ for every vertex $v\in V$ with infinite degree. Then, $G$ admits an unfriendly partition.
 \end{theo}
 
 \begin{corol}
 	\label{desigualdadetrivial}
 	$\aleph_{\omega}\leq \varkappa $.
 \end{corol}
\begin{proof}
	For instance, suppose that $\varkappa < \aleph_{\omega}$. Then, there is a graph $G = (V,E)$ with $|V|<\aleph_{\omega}$, whose vertices have infinite degree and that has no unfriendly partition. Fix $n\in\mathbb{N}$ so that $|V| = \aleph_n$. Since $G$ has no vertices of finite degree, $|N(v)| \in \{\aleph_0,\aleph_1,\dots,\aleph_n\}$. Noticing that every successor cardinal is regular, $G$ is under the hypothesis of Theorem \ref{finitoscardinais}. Therefore, contradicting its choice, $G$ has an unfriendly partition.  
\end{proof}

Consistently, it is also not hard to verify that the inequality above is tight. This because $\mathfrak{c}^{+\omega} = \aleph_{\omega}$ under the Continuum Hypothesis (CH), so that the construction of Milner and Shelah given by Theorem \ref{contraexemplozfc} attests that $\varkappa \leq \aleph_{\omega}$. On the other hand, as a less immediate conclusion, we will prove that the equality $\varkappa = \aleph_{\omega}$ is actually independent from $\mathrm{ZFC}+\aleph_{\omega}<\mathfrak{c}$, a theory in which the cardinals $\aleph_{\omega}$ and $\mathfrak{c}^{+\omega}$ are distinct. In particular, that statement is independent from the usual axioms of set theory. To that aim, we recover another theorem due to Milner and Shelah in \cite{shelah_milner_1990}:

 \begin{theo}[\cite{shelah_milner_1990}, Theorem 3]
 	\label{consistente}
 	 Suppose that there is a $p-$point generated by a family of size $\aleph_1$.  Then, there is a graph $G = (V,E)$ with $|V| = \aleph_{\omega}$, without unfriendly partitions and such that all its vertices have infinite degree. 
 \end{theo}

This result is convenient because its hypothesis is consistent with $\mathcal{ZFC}+\aleph_{\omega}<\mathfrak{c}$ (see Theorem 8.0 (b) of \cite{pponto2}, for instance). Therefore, Corollary \ref{desigualdadetrivial} and Theorem \ref{consistente} combined conclude that the statement $\varkappa = \aleph_{\omega}$ is consistent with $ZFC+\aleph_{\omega}<\mathfrak{c}$. Our goal now is prove that this equality is, actually, independent from these axioms.

Introducing the main tool we will apply to provide that proof, let $(\mathbb{P},\leq)$ be any partially ordered set. We say that a non-empty subset $F\subset \mathbb{P}$ is a \textbf{filter} if, roughly speaking, is a collection of big elements of $\mathbb{P}$. More precisely, it is a filter if the following properties are verified:
\begin{itemize}
	\item $F\neq \emptyset$.
	\item $F$ is upward-closed, that is, given $p\in \mathbb{P}$, $p\in F$ if $p \geq s$ for some $s\in F$.
	\item For any pair $s,t \in F$, there is $p\in F$ such that $p \leq s$ and $p \leq t$.
\end{itemize}

We also say that two elements $p,q\in \mathbb{P}$ are \textbf{incompatible} if there is no $r\in \mathbb{P}$ such that $r \leq p$ and $r \leq q$. On the other hand, $D\subset \mathbb{P}$ is called a \textbf{dense} subset in $\mathbb{P}$ if, for every $p\in \mathbb{P}$, there is $d\in D$ such that $d \leq p$. Finally, we say that $\mathbb{P}$ satisfies the \textbf{countable chain condition} if every collection of pairwise incompatible elements of $\mathbb{P}$ is countable. In that terms, Martin's Axiom is the following statement:
\begin{center}
	\textbf{Martin's Axiom (MA):} Fix $(\mathbb{P},\leq)$ a partially ordered set that satisfies the countable chain condition. If $\mathcal{D}$ is a family of dense subsets of $\mathbb{P}$ with $|\mathcal{D}| < \mathfrak{c}$, then there exists a filter $F$ of $\mathbb{P}$ such that $F\cap D \neq \emptyset$ for every $D\in\mathcal{D}$.
\end{center}

Convenient to our study, Martin's Axiom is independent of $\mathrm{ZFC}+\aleph_{\omega}<\mathfrak{c}$. Moreover, in a theory where $MA$ holds, $\mathfrak{c}$ is a regular cardinal. Using some other properties that can be consulted in \cite{livro}, for example, we are ready to prove the following:

\begin{prop}
    \label{martin}
Suppose that Martin's Axiom holds. Then, every graph $G = (V,E)$ having all its vertices of infinite degree and such that $|V|<\mathfrak{c}$ admits an unfriendly partition. 
\end{prop}
\begin{proof}
We will study the classical partial order $(\mathbb{P},\leq)$, where $\mathbb{P} = \{c: D \to 2 \mid D\subset V \text{ is finite}\}$ and $\leq$ is described by the following rule: $c \leq c'$ if, and only if, $c$ is an extension of $c'$. Denote by $\dom(c)$ the domain of a partial coloring $c\in \mathbb{P}$. It is well know that $(\mathbb{P},\leq)$ satisfies the countable chain condition. Moreover, it is easily verified that the items below define dense sets in $\mathbb{P}$:
\begin{enumerate} 
    \item For each $v\in V$, define $D_v = \{c\in \mathbb{P} : v\in\dom(c)\}$. 
    \item If $v\in V$ is a vertex of regular degree $\kappa_v$, fix $\{v_{\alpha}\}_{\alpha < \kappa_v}$ an enumeration of its neighborhood. Now, for each $\alpha < \kappa_v$, define the set $$R_v^{\alpha} = \{c\in \mathbb{P}: \text{ there are }\beta_0,\beta_1 > \alpha \text{ such that }c(v_{\beta_0}) = 0 \text{ and }c(v_{\beta_1}) = 1\}$$ 
    \item Similarly to the item above, if $v$ is a vertex of singular degree $\kappa_v$, fix $\{v_{\alpha}\}_{\alpha < \kappa_v}$ an enumeration of its neighborhood and $\{\gamma_{\xi}\}_{\xi < cf(\kappa_v)}$ a cofinal sequence in $\kappa_v$ of regular cardinals. Given $\alpha < \kappa_v$, let $\xi < cf(\kappa_v)$ be the index such that $\gamma_{\xi}\leq \alpha < \gamma_{\xi+1}$. Define then $$S_v^{\alpha} = \{c\in \mathbb{P}: \text{ there are }\alpha < \beta_0,\beta_1 < \gamma_{\xi+1} \text{ such that }c(v_{\beta_0}) = 0 \text{ and }c(v_{\beta_1}) = 1\}$$ 
\end{enumerate}

Therefore, the sets $\mathcal{D} = \{D_v : v \in V\}$, $\mathcal{R} = \{R_{v}^{\alpha} : v\in V \text{ has regular degree }\kappa_v, \alpha < \kappa_v\}$ and $\mathcal{S} = \{S_v^{\alpha}: v\in V \text{ has singular degree }\kappa_v, \alpha < \kappa_v\}$ are families of dense sets in $\mathbb{P}$. Regarding that $|\mathcal{D}| = |V|<\mathfrak{c}$, $|\mathcal{R}| \leq |V|\cdot |V| = |V| < \mathfrak{c}$ and $|\mathcal{S}|\leq |V|\cdot |V| = |V| < \mathfrak{c}$, Martin's Axiom guarantees the existence of a filter $F\subset \mathbb{P}$ that intersects every dense of $\mathcal{D}\cup \mathcal{R}\cup \mathcal{S}$.

We claim that the function $c= \displaystyle \bigcup_{f\in F}f$ is well-defined and that its domain is $V$. In fact, for every $v\in V$, once $F\cap D_v \neq \emptyset$, there is $f\in F$ with $v\in \dom(f)$. Moreover, if $g\in F$ is another coloring such that $v\in \dom(g)$, there is $h\in F$ satisfying $h \leq f,g$, because $F$ is a filter. Hence, $f(v) = h(v) = g(v)$, concluding the well definition of $c$.

Now, we will verify that $c$ is an unfriendly partition. For that, let $v\in V$ be any vertex and denote by $\kappa_v$ its degree. If $\kappa_v$ is regular, let $\{v_{\alpha}\}_{\alpha < \kappa_v}$ be the enumeration of its neighborhood as fixed by the item 2 above. By the choice of $F$, for every $\alpha < \kappa_v$ there is $f\in F \cap R_v^{\alpha}$. Then, by definition of $c$, $c(v_{\beta_0}) = f(v_{\beta_0}) = 0$ and $c(v_{\beta_1}) = f(v_{\beta_1}) = 1$ for some ordinals $\beta_0,\beta_1 > \alpha$. This proves that $\sup\{\alpha < \kappa_v : c(v_{\alpha}) = 0\} = \sup\{\alpha < \kappa_v : c(v_{\alpha}) = 1\} = \kappa_v$. As $\kappa_v$ is a regular cardinal, it follows that, in $c$, $v$ has $\kappa_v$ neighbors of color $0$ and $\kappa_v$ neighbors of color $1$. Therefore, $c$ is unfriendly in $v$.

Finally, suppose that $v$ is a singular cardinal. As done in the item 3 above, let $\{v_{\alpha}\}_{\alpha < \kappa_v}$ be an enumeration of its neighborhood and consider $\{\gamma_{\xi}\}_{\xi < cf(\kappa_v)}$ the cofinal sequence (of regular cardinals) fixed before. Then, given $\xi < cf(\kappa_v)$ and $\gamma_{\xi}\leq \alpha < \gamma_{\xi+1}$, again the choice of $F$ guarantees that there is $f\in F\cap S_{v}^{\alpha}$. This means that $c(v_{\beta_0}) = f(\beta_0) = 0$ and $c(v_{\beta_1}) = f(v_{\beta_1}) = 1$ for some $\alpha < \beta_0,\beta_1 < \gamma_{\xi+1}$. In other words, $\sup\{\gamma_{\xi} \leq \alpha < \gamma_{\xi+1} : c(v_{\alpha}) = 0\} = \sup\{\gamma_{\xi}\leq \alpha < \gamma_{\xi+1} : c(v_{\alpha}) = 1\} = \gamma_{\xi+1}$, implying that $|\{\gamma_{\xi} \leq \alpha < \gamma_{\xi+1} : c(v_{\alpha}) = 0\}|=|\{\gamma_{\xi}\leq \alpha < \gamma_{\xi+1} : c(v_{\alpha}) = 1\}|=\gamma_{\xi+1}$ by the fact that $\gamma_{\xi+1}$ is regular. Then, we proved that, in $c$, $v$ has at least $\gamma_{\xi+1}$ neighbors of color $0$ and at least $\gamma_{\xi+1}$ neighbors of color $1$, for every $\xi < cf(\kappa_v)$. Therefore, it has $\sup\{\gamma_{\xi+1}:\xi < cf(\kappa_v)\} = \kappa_v$ neighbors of each color. In particular, $c$ is unfriendly in $v$.  

\end{proof}

\begin{corol}
    \label{c1}
The statement $\varkappa = \aleph_{\omega}$ is independent from $\mathrm{ZFC}+\aleph_{\omega}<\mathfrak{c}$.
\end{corol}
\begin{proof}
	We already argued that $\varkappa = \aleph_{\omega}$ is consistent with $\mathrm{ZFC}+\aleph_{\omega}< \mathfrak{c}$. Verifying the consistence of its negative, we observe that $\varkappa \neq \aleph_{\omega}$ under $\mathrm{ZFC}+\mathrm{MA}+\aleph_{\omega}<\mathfrak{c}$. In fact, if $G = (V,E)$ is a graph without vertices of finite degree and $|V| = \aleph_{\omega}< \mathfrak{c}$, then, by Proposition \ref{martin}, $G$ admits an unfriendly partition. Therefore, $\varkappa \neq \aleph_{\omega}$ by the definition of $\varkappa$.   
\end{proof}

Now, we will study the consistence of the statement $\varkappa = \mathfrak{c}^{+\omega}$ and its negative under $\mathrm{ZFC}+\aleph_{\omega}<\mathfrak{c}$. That scenario is interesting because $\mathfrak{c}^{+\omega} \neq \aleph_{\omega}$ when supposing that $\aleph_{\omega}<\mathfrak{c}$. For this study, we revisit Theorem \ref{finitoscardinais}, due to Aharoni, Milner and Prikry, and present a slightly rewritten proof to provide a convenient statement.

Before that, given a partial coloring $c: D \to 2$ of a graph $G = (V,E)$ that has no vertices of finite degree, it is convenient to extend that function to the whole graph in an unfriendly way, that is, to be unfriendly in the set $V\setminus D$ of remaining vertices. To this aim, the \textbf{closure} of $c$ is the function $\overline{c}$ defined recursively as follows:
\begin{enumerate}    
	\item To start, consider $\overline{c}(v) = c(v) $ for every $c\in D$, so that $\overline{c}$ will be an extension of $c$. Let $D_0 = \{v\in V\setminus D : |N(v)\setminus D|< |N(v)|\}$ denote those vertices of $V\setminus D$ with less uncolored neighbors than their degree. In particular, for each vertex $v\in D_0$ there is a color $\overline{c}(v)\in 2$ such that $v$ has $|N(v)|$ neighbors of color $1-\overline{c}(v)$ within $D$. Writing $\overline{c}(v) = c(v)$ for all $v\in D$, this extends $\overline{c}$ to a coloring $\overline{c}: W_0 \to 2$ unfriendly in all the vertices of $D_0$, where $W_0 := D\cup D_0$.

	\item For each ordinal $\alpha > 0$, consider that $W_{\beta}$ and $\overline{c}(v)$ are defined for each $v\in W_{\beta}$ with $\beta < \alpha$. As done with $D_0$, let $$D_{\alpha} = \left\{v\in V\setminus \left(\bigcup_{\beta < \alpha}W_{\beta}\right): \left|N(v) \setminus \left(\bigcup_{\beta < \alpha}W_{\beta}\right)\right|< |N(v)|\right\}$$ be the set of uncolored vertices of infinite degree whose neighbors are almost all colored. Then, for each $v\in D_{\alpha}$ there is a color $\overline{c}(v)\in 2$ such that $|N(v)|$ of its neighbors have color $1-\overline{c}(v)$. As before, this defines $\overline{c} $ to be a partial coloring unfriendly in $D_{\alpha}$, and its domain will now be denoted by $\displaystyle W_{\alpha} = D_{\alpha}\cup \bigcup_{\beta < \alpha}W_{\beta}$.
	
	\item  If $\Gamma$ is the least ordinal such that $D_{\Gamma} = \emptyset$, this procedure defines a coloring $\overline{c}: \displaystyle \bigcup_{\alpha < \Gamma}W_{\alpha} \to 2$ that extends $c$ and that is unfriendly in all the vertices of $W_{\alpha}\setminus D$ for every $\alpha < \Gamma$. Denoting its domain by $\overline{D}$, we observe that, by the choice of $\Gamma$, $|N(v)\cap (V\setminus \overline{D})| = |N(v)|$ for all $v\in V\setminus \overline{D}$. In particular, $\overline{\overline{c}} = \overline{c}$.
\end{enumerate}

Inspired by the fact that $\overline{\overline{c}} = \overline{c}$, we say that a partial coloring $c: D \to 2$ is \textbf{closed} if $\overline{c} = c$. Considering this definition, we are ready to prove the statement below:    

\begin{theo}[\cite{Aharoni1990}, Theorem 2]
    \label{adaptado}
Let $\mathcal{K}$ be a family of infinite cardinals such that the following property holds: 
\begin{center}
Every graph whose vertices have degree as a cardinal of $\mathcal{K}$ has an unfriendly partition.
\end{center}
Then, if $\kappa$ is a regular cardinal greater than every member of $\mathcal{K}$, every graph whose vertices have degree as a cardinal of $\mathcal{K}\cup \{\kappa\}$ admits an unfriendly partition.
\end{theo}
\begin{proof}
Let $G = (V,E)$ be a graph with $|N(v)|\in \mathcal{K}\cup \{\kappa\}$ for every $v\in V$. To better apply our hypothesis, let $M\subset V$ be the set of vertices of degree $\kappa$ and $N = V\setminus M$ be the set of vertices of degree in $\mathcal{K}$. For any subset $X\subset V$, we will write $N(X) = \displaystyle \bigcup_{v\in X}N(v)$. Then, the following claim holds:
\begin{center}
    \textbf{Claim:} $|N(C)|< \kappa$ for each connected component $C = (V_C,E_C)$ of $G[N]$.
\end{center}
\begin{proof}[Proof of the claim]
Fix a vertex $v\in N(C)$ and, for each $i\in\mathbb{N}$, denote by $N^i(v) = \{x\in N: \text{ there is a path of length }i\text{ from }v\text{ to }x\text{ contained in }G[N]\}$ the set of vertices of $N$ connected to $v$ by a path of $i$ vertices of $G[N]$. Therefore, the definition of connected component guarantees that $V_C = \displaystyle \bigcup_{i\in \mathbb{N}}N^i(v)$. Besides that, $N^1(v) = N(v)$ has less than $\kappa$ vertices once $v\in N$ and, for each $i \in \mathbb{N}$, $N^{i+1}(v)\subset \displaystyle\bigcup_{x \in N^i(v)}N^i(x)$. Supposing that $|N^i(v)|< \kappa$, we conclude that $|N^{i+1}(v)|<\kappa$ since $\kappa$ is a regular cardinal. In other words, by induction on $i$, we proved that $|N^i(v)|<\kappa$ for every $i\in\mathbb{N}$ and, then, again by the fact that $\kappa$ is a regular cardinal, $|V_C|< \kappa$.

Regarding that $|N(v)|< \kappa$ for every $v\in V_C\subset N$ and that $N(C) = \displaystyle \bigcup_{v\in V_C}N(v)$, it follows that $|N(C)|<\kappa$, one more time by the fact that $\kappa$ is a regular cardinal.
\end{proof}

For simplicity, to the end of this proof, we will denote the vertex set $V_C$ of a connected component $C$ of $G[N]$ also by $C$. To properly color some vertices of $G$, we say that two (disjoint) sets $F_0\subset M$ and $F_1\subset N$ define a \textbf{bipartite pair} $(F_0,F_1)$ if, for every $v\in F_0$ and $u\in F_1$, we have $|N(v)\cap F_1| = |N(v)| = \kappa$ and $|N(u)\cap F_0| = |N(u)| <\kappa$. In other words, for $i\in\{0,1\}$, every member of $F_i$ has its degree of neighbors in $F_{1-i}$. Note that this property is closed by unions, because, if $(F_0',F_1')$ is another bipartite pair, than every member of $F_i\cup F_i'$ has its degree as amount of neighbors in $F_{1-i}\cup F_{1-i}'$. Then, we denote by $(F_0,F_1)$ the \textbf{maximal bipartite pair}, described by the unions $F_0 = \displaystyle \bigcup_{(F_0',F_1')\in \mathbb{BP}}F_0'$ and $F_1 = \displaystyle \bigcup_{(F_0',F_1')\in \mathbb{BP}}F_1'$, where $\mathbb{BP}$ is the set of all bipartite pairs of $G$.

That notation induces a natural unfriendly partial coloring $c': F_0 \cup F_1 \to 2$, given by $c'(v) = 0$ and $c'(u) = 1$ for every $v\in F_0$ and $u\in F_1$. Denoting its closure by $\overline{c}': \overline{D}\to 2$, it follows that $\overline{c}'$ is also an unfriendly partial coloring and that $|N(v)\setminus \overline{D}| = |N(v)|$ for every $v\in V\setminus \overline{D}$. Therefore, it is enough to find an unfriendly partial coloring $c : V\setminus \overline{D}\to 2$ to $G[V\setminus \overline{D}]$, so $c\cup \overline{c}'$ will be the unfriendly partition desired. To this end, the choice of the pair $(F_0,F_1)$ guarantees the property below:

\begin{center}
\textbf{Claim:} Let $S\subset M \setminus \overline{D}$ be any set with $|S|<\kappa$. Then, the set $T = \{u\in N\setminus \overline{D}: |N(u)\cap S| = |N(u)| \}$ has fewer than $\kappa$ vertices.
\end{center}

\begin{proof}[Proof of the claim]
Define the sets $A = \{v\in S : |N(v)\cap T|< |N(v)| = \kappa\}$ and $B = \displaystyle \{u\in T: u \in N(x)\text{ for some }x\in A\} = \bigcup_{x\in A}(N(x)\cap T)$. Once $A\subset S$, we have that $|A|<\kappa$. Then, $|B|< \kappa$ by definition of $A$ and the regularity of $\kappa$. Noticing that $(S\setminus A, T\setminus B)$ is a bipartite pair, we conclude that $T\setminus B = \emptyset$ by the fact that the maximal bipartite pair $(F_0,F_1)$ is already colored. Therefore, $|T| = |B|<\kappa$.  
\end{proof}

To construct the desired coloring of $G[V\setminus \overline{D}]$, fix a non-injective enumeration $\{v_{\alpha}\}_{\alpha < \kappa}$ of $M\setminus \overline{D}$ such that every member is presented $\kappa$ times, i.e., $|\{\alpha < \kappa: v_{\alpha} = v\}| = \kappa$ for every $v\in M\setminus \overline{D}$. Then, we will define recursively an unfriendly partition $c: V\setminus \overline{D}\to 2$ according to the following algorithm:

\begin{enumerate}
    \item We first define $c(v_0) = 0$. If $v_0$ has a neighbor $v\in M\setminus \overline{D}$, we define $c(v) = 1$. If not, once $|N(v_0)\setminus \overline{D}| = |N(v_0)| = \kappa$, there is a component $C$ of $G[N\setminus \overline{D}]$ where $v_0$ has a neighbor $v$. For every $u\in N(C)\cap M \setminus (\overline{D}\cup \{v_0\})$, we set $c(u) = 0$. By sewing $c$ as a coloring defined in $G[C\cup N(C)\setminus \overline{D}]$, we extend it to some vertices $\overline{C}\subset C$ by taking its closure. Then, a vertex $u\in C\setminus \overline{C}$ satisfies $|N(u)\cap (C\setminus \overline{C})| = |N(u)|\in \mathcal{K}$. Hence, by hypothesis, we may extend $c$ to the whole component $C$ by adjoining an unfriendly partition of $G[C\setminus \overline{C}]$. Up to changes of the colors of all the vertices of $(C\cup N(C))\setminus (\overline{D}\cup \{v_0\})$, we can assume that $c(v) = 1$. This finishes the first iteration of the algorithm. Note that, besides those vertices of $\overline{D}$, we have colored a component of $G[N\setminus \overline{D}]$ and less than $\kappa$ vertices of $M$, as the first claim guarantees.
    \item For some ordinal $\alpha > 0$, denote by $S_{\alpha}\subset M\setminus \overline{D}$ the vertices of $V\setminus \overline{D}$ of degree $\kappa$ that we have colored so far by this algorithm. Since $\alpha < \kappa$ and $\kappa$ is a regular cardinal, by transfinite induction we can suppose that $|S_{\alpha}|<\kappa$. Then, first consider the case that $c(v_{\alpha})$ is already defined. If $v_{\alpha}$ has a neighbor $v\in M\setminus (\overline{D}\cup S_{\alpha})$, define $c(v) = 1-c(v_{\alpha})$. If not, then $v_{\alpha}$ has $\kappa$ neighbors as elements of $N\setminus \overline{D}$. Moreover, once each connected component of $G[N\setminus \overline{D}]$ has cardinality less than $\kappa$, $v$ has neighbors in $\kappa$ distinct of them. Regarding that less than $\kappa$ of such components were colored so far, by the last claim we may find a connected component $C$ of $G[N\setminus \overline{D}]$ such that $|N(u)\cap S_{\alpha}|< |N(u)|$ for every vertex $u$ of $C$. In other words, every member of $C$ has in $V\setminus \overline{D}$ as many neighbors as in $D_{\alpha}: = V\setminus (\overline{D}\cup S_{\alpha})$. Similarly to the procedure of the first iteration, define $c(u) = 0$ for each $u\in N(C)\cap M \cap D_{\alpha}$. Regarding $c$ as a partial coloring of the graph $G[C\cup (N(C)\cap D_{\alpha})]$, we define $c$ on some vertex subset $\overline{C}$ of $C$ by taking its closure. Hence, every remaining vertex $u\in C\setminus \overline{C}$ is such that $|N(u)\cap (C\setminus \overline{C})| = |N(u)|\in \mathcal{K}$. Now, the hypothesis can be employed to extend $c$ to $G[C\setminus \overline{C}]$ by adjoining an unfriendly partition of such subgraph. Again, up to changing the colors of all the vertices of $C\cup (N(C)\cap D_{\alpha})$, we can assume that $c(v) = 1-c(v_{\alpha})$. Finally, if $c(v_{\alpha})$ was not defined, just set $c(v_{\alpha}) =0$.      
\end{enumerate}

At the end of this transfinite process, $c$ is defined for every vertex of $M$ and for some connected components of $G[N]$. In such components, this is an unfriendly coloring: each vertex was colored by some closure throughout the procedure or by an unfriendly partition given by the hypothesis. A vertex of $M\setminus \overline{D}$, instead, received a neighbor of opposite color in $\kappa$ iterations before the moment his value by $c$ has been set, according to its indices at the enumeration $\{v_{\alpha}\}_{\alpha < \kappa}$. Therefore, every member of $M\setminus \overline{D}$ has $\kappa$ neighbors of opposite color, verifying that $c$ is unfriendly in those vertices.

It remains, however, to define colors for vertices of some components of $G[N\setminus \overline{D}]$ that were not analyzed by the steps above. If $C$ is one of those components, as before we will see the coloring $c$ as a partially defined coloring at $G[C\cup N(C)]$. In fact, it is only defined in $N(C)\cap M$. By taking its closure within this component, $c$ is defined for some vertex set $\overline{C}\subset C$ and, therefore, it is unfriendly in such vertices. Again, every member $u\in C\setminus \overline{C}$ satisfies $|N(u)\cap (C\setminus\overline{C})| = |N(u)|\in \mathcal{K}$. Hence, the hypothesis can be applied to define $c$ as an unfriendly partition for $G[C\setminus \overline{C}]$, coloring the entire component $C$.    

\end{proof}

\begin{corol}
    \label{c2}
The statement $\varkappa = \mathfrak{c}^{+\omega}$ is independent from $\mathrm{ZFC}+\aleph_{\omega}<\mathfrak{c}$.
\end{corol}
\begin{proof}
By Martin's Axiom, that is consistent with $\mathrm{ZFC}+\aleph_{\omega}<\mathfrak{c}$, every graph without vertices of finite degree and less than $\mathfrak{c}$ vertices admits an unfriendly partition, as guaranteed by Proposition \ref{martin}. Since $\mathfrak{c}$ is a regular cardinal under Martin's Axiom, by Theorem \ref{adaptado} we actually have that graphs with at most $\mathfrak{c}$ vertices, all them of infinite degree, have unfriendly partitions. Denote by $\alpha > \omega$ the ordinal such that $\mathfrak{c} = \aleph_{\alpha}$. Observing that successor cardinals are regular, Theorem \ref{adaptado} also shows that every graph with at most $\aleph_{\alpha+n}$ vertices, all them of infinite degree, has an unfriendly partition. This means that a graph without unfriendly partitions and vertices of finite degree must have at least $\mathfrak{c}^{+\omega} = \aleph_{\alpha+\omega}$ vertices. Concluding that the statement $\varkappa = \mathfrak{c}^{+\omega}$ is consistent with $\mathrm{ZFC}+\aleph_{\omega}<\mathfrak{c}$, the construction of Milner and Shelah given by Theorem \ref{contraexemplozfc} has such size.

On the other hand, the axioms used by Milner and Shelah to prove Theorem \ref{consistente} are also consistent with the statement $\aleph_{\omega}<\mathfrak{c}$. So it is consistent with $\mathrm{ZFC}+\aleph_{\omega}<\mathfrak{c}$ that there is a graph with less than $\mathfrak{c}^{+\omega}$ vertices, all of them of infinite degree, that does not admit an unfriendly partition. That is, the statement $\varkappa \neq \mathfrak{c}^{+\omega}$ is also consistent with $\mathrm{ZFC}+\aleph_{\omega}<\mathfrak{c}$.

Hence, the statement $\varkappa = \mathfrak{c}^{+\omega}$ is independent of $\mathrm{ZFC}+\aleph_{\omega}<\mathfrak{c}$.
\end{proof}

Therefore, together, Corollaries \ref{c1} and \ref{c2} prove Theorem \ref{nosso}. In other words, if we want to forbid vertices of finite degree, the minimum size of a graph without unfriendly partitions can't be precisely determined, even knowing that the cardinals $\aleph_{\omega}$ and $\mathfrak{c}^{+\omega}$ are distinct.

\section{Acknowledgments}
\paragraph{}
Both authors thank the financial support of FAPESP. The first named author was supported through grant numbers 2019/22344-0 and 2023/00595-6, while the second named author was supported through grant number 2021/13373-6.  

\bibliographystyle{plain}
\bibliography{UnfriendlyPartitionsWhenAvoidingFiniteDegree}

\end{document}